\documentclass[reqno,12pt]{amsart}
\usepackage{amssymb,amsmath,amsthm,amsfonts,color}

\usepackage{mathrsfs,dsfont, comment,mathscinet}
\usepackage{todonotes}

\usepackage{enumerate,esint}
\usepackage[left=2.8cm,right=2.8cm,top=2.8cm,bottom=2.8cm,marginparwidth=7mm,marginparsep=1mm]{geometry}
\usepackage{mathtools}
\usepackage{accents}
\usepackage{graphicx}
\usepackage{bbm}
\usepackage[abs]{overpic}
\usepackage{tikz} 
\usetikzlibrary{patterns}
\usetikzlibrary{decorations.fractals}
\usetikzlibrary{arrows}
\usepackage{xcolor} 
\usepackage{enumitem}
\usepackage{pgf}

\usepackage{epstopdf}
\usepackage[colorlinks=true, pdfstartview=FitV, linkcolor=blue, 
            citecolor=blue, urlcolor=blue]{hyperref}

\usepackage[normalem]{ulem}

\allowdisplaybreaks

\theoremstyle{plain}
\newtheorem{Theorem}{Theorem}[section]
\newtheorem{Proposition}[Theorem]{Proposition}
\newtheorem{Lemma}[Theorem]{Lemma}

\theoremstyle{definition}

\newtheorem{Definition}[Theorem]{Definition}

\numberwithin{equation}{section}

\DeclareMathOperator*{\dist}{dist}
\DeclareMathOperator*{\capacity}{Cap}

\newcommand{\loc}{\textrm{loc}}

\newcommand{\mN}{{\mathbb N}}
\newcommand{\mR}{{\mathbb R}}

\newcommand{\cA}{{\mathcal A}}

\newcommand{\cC}{{\mathcal C}}
\newcommand{\cF}{{\mathcal F}}

\newcommand{\cH}{{\mathcal H}}

\newcommand{\cN}{{\mathcal N}}

\newcommand{\oomega}{{\bar{\omega}}}

\newcommand\R{\mathbb R}

\newcommand\M{\mathbb M}
\newcommand\N{\mathbb N}

\newcommand\ep{\varepsilon}

\newcommand\wk{\rightharpoonup}
\renewcommand\d{\,\mathrm{d}}

\newcommand\UUU{\color{black}}
\newcommand\EEE{\color{black}}

\newcommand\MMM{\color{black}}
\newcommand\OOO{\color{black}}

\title[\UUU Equilibria of charged hyperelastic solids]{Equilibria \UUU of charged hyperelastic solids\EEE}

\author[E. Davoli]{Elisa Davoli}
\address{Institute for Analysis and Scientific Computing, TU Wien, Wiedner Hauptstra{\ss}e 8--10, 1040 Wien, Austria}
\email{elisa.davoli@tuwien.ac.at}

\author[A. Molchanova]{Anastasia Molchanova}
\address{Institute for Analysis and Scientific Computing, TU Wien, Wiedner Hauptstra{\ss}e 8--10, 1040 Wien, Austria}
\email{anastasia.molchanova@tuwien.ac.at}

\author[U. Stefanelli]{Ulisse Stefanelli}
\address[Ulisse Stefanelli]{Faculty of Mathematics, University of
  Vienna, Oskar-Morgenstern-Platz 1, A-1090 Vienna, Austria,
Vienna Research Platform on Accelerating
  Photoreaction Discovery, University of Vienna, W\"ahringerstra\ss e 17, 1090 Wien, Austria,
 \& Istituto di
  Matematica Applicata e Tecnologie Informatiche {\it E. Magenes}, via
  Ferrata 1, I-27100 Pavia, Italy
}
\email{ulisse.stefanelli@univie.ac.at}


\begin{document}

\begin{abstract}
  \UUU We investigate equilibria of charged deformable materials via
  the minimization of an electroelastic energy. This \MMM features \UUU the coupling \MMM of \UUU elastic response and
  electrostatics by means of a capacitary term, which is
  naturally defined in Eulerian coordinates. The ensuing electroelastic energy
  is then of mixed \OOO Lagrangian--Eulerian \UUU type. We prove that  minimizers
  exist by investigating  the
  continuity properties of the capacitary terms
  under convergence of the deformations. \EEE
\end{abstract}

\maketitle


\section{Introduction}

\UUU

The interaction of electric and mechanical effects in solids is
crucial in different modeling situations and is at the
basis of a variety of applications. We focus here on the description
of the electromechanical equilibrium of a charged
conductor. In \MMM the \UUU absence of an external electric field, this results from the interplay of mechanical  and electrostatic response. The first favors specific deformations under the effect of given external mechanical
loads, whereas the latter favors shapes of
larger capacitance. 

More specifically, we consider the case of a hyperelastic charged conductor, embedded
in \MMM an \UUU insulating medium, which we also assume to be deformable. This
setting is inspired to electroactive-polymer
devices, featuring indeed conductive parts embedded in polymeric
matrices \cite{Ortigosa0}. Coated
wires, printed circuit boards, and  capacitive deformation
sensors \cite{Wang} also \MMM fit within this framework. \UUU

The actual configuration of the 
whole system of conductor and insulator is specified by its deformation
\OOO
$y\colon  \Omega \to \R^d$ 
\UUU
from the bounded reference configuration $\Omega
\subset \R^d$, $d\geq 2$. More precisely, we indicate the reference
configuration of the conductor by $\omega \subset \Omega$, see Figure
\ref{fig1}, so that $y(\oomega)$ indicates the actual position of the deformed
conductor whereas $y(\bar{\Omega}\setminus \oomega)$ indicates the
deformed insulator.

We assume that the conductor carries a given total charge
$Q$. Its equilibrium results
from a competition of mechanical and electric actions. On the one
hand, the body may be subjected to mechanical loading, favoring specific
deformations. On the other hand, the actual shape of the body
determines its electrostatic potential with respect to the background potential. In particular, deformations $y$ maximizing the
{\it electric
  capacitance}  of the \UUU deformed
shape $y(\oomega)$  are 
\OOO preferred. \UUU
A first realization of this
competition, is encoded in the following choice for the  {\it electroelastic} stored energy
of the system 
\begin{equation}\label{def:energy}
    \cF_1(y):=\int\limits_{\Omega}W(\UUU x,\EEE\nabla y(x))\d x+
    \frac{\UUU Q^2}{2\UUU \capacity(y(\oomega))},
  \end{equation}
resulting indeed from the sum of the
{\it elastic stored energy} and the {\it electrostatic
  potential}. 

\EEE

\begin{figure}[h]
\label{fig1}
\centering
\begin{tikzpicture}[xscale=0.7,yscale=0.7, >=latex]
    \draw [ultra thick] (-1,-0.2) -- (-1, 3.5);
    \draw (-1,-0.2) -- (5, -0.2) -- (5,3.5)-- (-1, 3.5);

    \draw [ultra thick] (8,3.5) -- (8,-0.2);
    \draw plot [smooth,tension=0.8] coordinates {(8,-0.2) (9,0.1) (11.5,-0.2) (13,4) (9.5, 3.5) (8,3.5)};
    \draw plot [smooth cycle,tension=0.8] coordinates {(-0.1,0.8) (0.5,0.3) (1.5,1.2) (3,0.7) (3.4,1.5) (2.4,2.3) (0.6,2)};
    \fill [blue, opacity=0.2] plot [smooth cycle,tension=0.8] coordinates {(-0.1,0.8) (0.5,0.3) (1.5,1.2) (3,0.7) (3.4,1.5) (2.4,2.3) (0.6,2)};
    \draw plot [smooth cycle,tension=0.8] coordinates {(9.4,0.5) (10.5,0.5) (11.9,1.7) (10.8,2.7) (9.4,1.6)};
    \fill [blue, opacity=0.2] plot [smooth cycle,tension=0.8] coordinates {(9.4,0.5) (10.5,0.5) (11.9,1.7) (10.8,2.7) (9.4,1.6)};

    \draw [->, blue, ultra thick] (3,3.0) to [out=45,in=135] (9,3);

    \node at (-0.3,2.8) {$\bar \Omega$};
    \node at (-1.5,0.4) {$\Gamma_0$};
    \node at (7.5,0.4) {$\Gamma_0$};
    \node at (2.7,1.5) {$\bar \omega$};
    \node at (12.1,3.2) {$y(\bar \Omega)$};
    \node at (10.9,1.7) {$y(\bar \omega)$};
    \node at (6,4.7) {$y$};  

\end{tikzpicture}
\caption{Setting of the problem.} 
\end{figure}
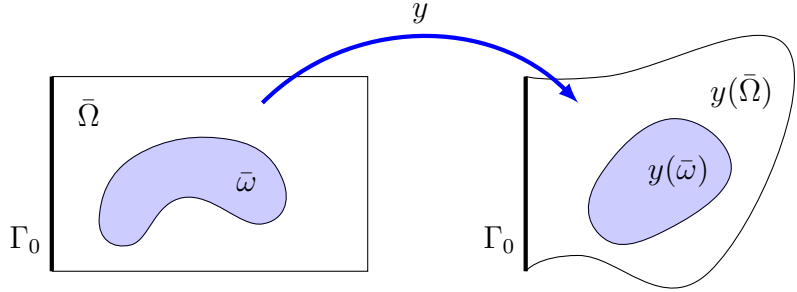

\UUU
In the expression above, 
$W\colon\UUU \Omega \times \EEE \M^{d\times d}\to [0,+\infty)$ is \UUU 
the elastic \EEE energy density \UUU of the medium, where $\M^{d\times
  d}$ indicates $(d\times d)$-matrices. It is
assumed to be a Carath\'eodory integrand and to satisfy \EEE the following
assumptions:
\begin{align}\UUU
\text{(Polyconvexity)} \qquad  &W(x,F)={\mathbb
                W}(x,\mathcal{M}(F)) \ \ \text{where $\mathbb W(x,\cdot)$ is convex
                }\nonumber\\
  & \text{in the minors $\mathcal{M}(F)$ of $F\in \mathbb
                M^{d\times d}$, $\forall x \in \Omega$} \label{hp:poly}\\[2mm]
\text{(Growth)}\qquad  &
                                                    \text{\UUU there exists
                                                    $c_W>0$, $p>d$, and
                                                    $s>d-1$ such that}
                                  \nonumber\\
  &
 \begin{cases}
W(\UUU x, \EEE F)\geq \UUU c_W\EEE |F|^q
+ \UUU c_W\EEE \dfrac{|F|^{ds}}{{\det F}^s}\UUU -\displaystyle\frac{1}{c_W} \EEE &\
\     \text{if}\ \textrm{det}\,F\geq 0,\\
W(\UUU x, \EEE F)=+\infty&\text{otherwise},
\end{cases}\label{hp:growth}
                           \nonumber\\[1mm]
&\text{$\forall F \in \M^{d\times d} $, \
                       $\forall x \in \Omega$.
                                  }
\end{align}
\UUU In particular, the space-dependence $x\in \Omega \mapsto W(x,F)$
models the possibly very different elastic response of the conductor
and the insulator.  \EEE
We
refer to Section \ref{sec:adm_def} below for a discussion on the
meaning and role of the second term in the first line of
\eqref{hp:growth}. \UUU Let us mention that  the above
conditions 
\OOO 
\eqref{hp:poly}--\eqref{hp:growth} 
\UUU
are compatible with {\it
  frame indifference}, namely, $ W(\hat RF)=W(F) $ for every 
rotation $\hat R\in
                SO(d):=\{ R \in \M^{d\times d}: \, R^{-1}=R^T \
                \text{and} \ \det R=1\}$.  Although not 
                needed in our analysis, frame indifference is a crucial requirement from
                the modeling viewpoint. Note \MMM that \UUU no loads are assumed for the sake of
notational simplicity. The case of nonvanishing loads can be treated
as well. \EEE

\UUU As concerns the electric capacitance of the deformed body
$y(\oomega)$, the choice in \eqref{def:energy} corresponds to that of
{\it self-capacitance}, where the electrostatic potential is taken in
relation with a far background, ideally at $\infty$. From the modeling
viewpoint, this corresponds to the case in which the complement of the
deformed body $y(\oomega)$ has negligible dielectric response.
In $d \geq 3$
dimensions, for all compact sets $E\subset \R^d$ one defines
$$\frac{1}{\capacity(E)} =
\min_{\mu}\int\limits_{E}\int\limits_{E}\frac{1}{\sigma_d \varepsilon_0|x-y|^{d-2}}\d \mu
  (x)\, \d \mu(y)$$
where $\sigma_d$ is the surface of the unit ball, $\varepsilon_0$ denotes the permittivity of vacuum, and the
infimum is taken on nonnegative Borel measures $\mu$ with support
in the compact set $E \subset \R^d$ and $\mu(E)=1$. Note that such minimum exists and concentrates on the
boundary 
$\partial E$. Following 
\cite[Ch.~11.15]{LieLos2001}, one can equivalently variationally
reformulate the latter by letting the {\it capacity} of $E$ be defined as
\begin{align}
    & \nonumber
    \capacity(E):=\inf\Bigg\{\int\limits_{\R^d}|\nabla
      v(\xi)|^2\d\xi:\,v\in L^{2^*}(\R^d)\text{ with }\nabla v\in L^2(\R^d;\R^d)\\
      &\qquad\qquad\qquad\qquad\qquad\qquad\qquad \label{eq:P}\text{ and }v\geq 1\text{ a.e.~in a neighb. of }E\Bigg\}
\end{align}
where $2^* = \frac{2d}{d-2}$, see \cite[Ch.~4.3, 8.2, and 8.3]{LieLos2001}.

  \EEE 
  \UUU The aim of this paper is to investigate minimizers of the
  electroelastic energy $\cF_1$ on the set of {\it admissible
    deformations}
  \begin{equation}
    \cA:=\{y\in W^{1,q}(\Omega;\mR^d) :\ y \ \text{is a homeomorphism
  and} \ y={\rm id} \  \text{on} \ \Gamma_0\}\label{cA}
\end{equation}
where the boundary portion $\Gamma_0 \subset \partial \Omega$ where
the body is clamped is assumed to be nonempty and open in the topology
of $\partial \Omega$. It is worth noting that the electroelastic
energy $\cF_1$ is 
of {\it mixed Eulerian-Lagrangian} type. Indeed, deformations are Lagrangian in nature,
for they relate to the reference configuration, whereas the capacitary
term is Eulerian, as it depends on the actual shape $y(\oomega)$ of
the conductor. \EEE


\UUU 

We are also interested in a second, different situation where the reference
electrostatic potential is that of the complement of the deformed
conductor-insulator system $y(\bar \Omega)$. This \MMM framework \UUU corresponds to
the case when the complement $\R^d \setminus y(\bar \Omega) $ is assumed
to be 
conductive, which specifically refers to the setting of {\it
  capacitors}. Here, the relevant notion is that of {\it relative (or mutual)
  capacity} of two conductors $E$ and $\R^d\setminus D$, where $E$ is
compact and $D$ is open and contains $E$, which is specified as 
\begin{align}
    & 
    \capacity(E;D):=\inf\Bigg\{\int\limits_{D}|\nabla
      v(\xi)|^2\d\xi:\,v\in W^{1,2}_0(D)\label{eq:P2}\text{ and }v\geq 1\text{ a.e.~in a neighb. of }E\Bigg\}.
\end{align}
In this setting, we consider the electroelastic energy
 $$   \cF_2(y):=\int\limits_{\Omega}W(\UUU x,\EEE\nabla y(x))\d x+
 \frac{\UUU Q^2}{2\capacity(y(\oomega);y(\Omega))},$$
 again to be minimized on the class of admissible deformations $\cA$.


Our main result reads as follows.
\begin{Theorem}[\UUU Existence of equilibria]
\label{thm:main} \UUU $\cF_1$ and $\cF_2$ admit minimizers in $\cA$. \EEE
\end{Theorem}

The proof of Theorem \ref{thm:main} \UUU is given in Section \ref{sec:proof-main}
below and \EEE hinges upon two main ingredients: \UUU (1) a {\it closure}
property of \EEE admissible
deformations \UUU $\cA$ \EEE and (2) \EEE an {\it upper-semicontinuity} \UUU
result \EEE for
the \UUU capacitary terms  \EEE under \UUU the \EEE  uniform
convergence of the deformations. \UUU These two ingredients in particular allow to
apply the Direct Method to $\cF_1$  and $\cF_2$ and secure the
existence of  minimizers.

\UUU For the sake of completeness, we also provide a {\it
  lower-semicontinuity} result for the capacitary terms so that, ultimately,
 \begin{align}
    &  \capacity(y(\bar{\omega}) )=\lim_{n\to+\infty}
      \capacity(y^n(\bar{\omega}) )  \ \text{and}\ 
      \capacity(y(\bar{\omega}); y(\Omega))=\lim_{n\to+\infty} \capacity(y^n(\bar{\omega}); y^n(\Omega)) \label{eq:cont_capacity}
      \end{align}
      whenever $y\colon \Omega \to \mR^d$
    and $\{y^n\}\subset C^0(\bar{\Omega};\mR^d)$
    are homeomorphisms such that $y^n\to y$ strongly in $C^0(\bar{\Omega};\mathbb{R}^d)$.
 Note however that lower
semicontinuity holds under some specific geometrical constraints on $y(\Omega)$.
We refer to Propositions \ref{prop:limsup-cap} and \ref{prop:lsc_cap}
below for the precise statements. 
The main technical tool for the proof is a detailed characterization
of the monotonicity behavior \UUU of the \EEE  capacity with respect
to its arguments, cf. Proposition~\ref{prop:cap_prop}.

\UUU Before closing this introduction, let us remark that existence
results in the setting of electroelastostatics are not new. The
equilibrium of an electromagnetoelastic polyconvex material in void
has been already investigated in the 
\cite{Silhavy1}. The main tool there is the careful use of
{A}-quasiconvexity, related to relaxation
under linear   differential
constraints. In this specific case, such constraints naturally correspond to the static Maxwell equations. 
A sufficient condition for the polyconvexity of isotropic  electromagnetoelastic
energy densities is given in 
\cite{Silhavy2}. In contrast with our setting, the formulation in
\cite{Silhavy1,Silhavy2} is purely Lagrangian and no charge is
considered.

The variational modelization in \cite{Ortigosa0,Ortigosa1,Ortigosa2} moves
along the same lines of \cite{Silhavy1}, allowing charges and assuming
the conductor to be surrounded by a polymeric matrix, as in our
case. Let us note however, that the focus there is on modelization and
simulation. In particular, no 
existence result for equilibria is provided.

In the series of papers \cite{Novaga0,Novaga00,Novaga1,Novaga2} the
authors analyze the 
equilibrium shape of two-dimensional charged, perfectly conducting
liquid drops. There, a variational energy of the type of $\cF_1$ is studied, where
nonetheless the elastic part is replaced by the perimeter of the
liquid drop. Under different settings, existence for the corresponding
minimization problem may hold or fail in different classes of shapes.

As already mentioned, our variational model is of mixed 
\OOO
Eulerian--Lagrangian
\UUU
type, a class which has recently attracted attention due to its
relevance in connection with multiphysics applications. Without any
claim of completeness, let us recall that the mathematical analysis
of mixed
\OOO
Eulerian--Lagrangian
\MMM formulations \UUU have been considered in the modelization of
defective crystals \cite{dacorogna.fonseca,fonseca.parry}, in the
setting of nematic elastomers \cite{barchiesi.desimone,
  barchiesi.henao.moracorral}, in dislocation-free finite plasticity
\cite{stefanelli, kruzik.melching.stefanelli}, in bulk-damage modeling
\cite{davoli.kruzik.pelech}, and in magnetostriction
\cite{rybka.luskin, kruzik.stefanelli.zeman}. \EEE Dimension reduction in nonlinear magnetoelasticity has been studied analytically and numerically in \cite{liakhova, liakhova.luskin.zhang, luskin.zhang} under further restrictions on the 
\OOO 
Jacobian
\EEE 
of elastic deformations. The membrane and Von K\'arm\'an regimes are the subject of \cite{davoli.kruzik.piovano.stefanelli} and \cite{bresciani}, respectively. For energy functionals featuring both bulk and surface terms, as well as for refined phase-field models, we refer to \cite{javili.mcbride.steinmann}, \cite{levitas}, and to the two recent contributions \cite{GraKruMaiSte2019, GraKruMaiSte2019-2}.

\EEE

Our paper is organized as follows. In Section \ref{sec:prel}, we
introduce notation and recall results on Sobolev spaces with zero
traces and on smooth approximations of sets. Section \ref{sec:adm_def}
\UUU discusses the properties of \EEE admissible deformations, as well
as a connection with the theory of mappings with finite
distortion. Section \ref{sec:capacity} \MMM analyzes upper semicontinuity
of the capacity and Section \ref{sec:proof-main} contains the proof of Theorem \ref{thm:main}. Eventually, Section \ref{sec:capacity-new} completes our study of continuity properties for capacitary terms and provides a discussion on the geometry of deformed sets, cf. Subsection \ref{subs:geo}. \EEE


\section{\UUU Notation and preliminaries}
\label{sec:prel}


In this section, we collect definitions, notation, and preliminary
results which will be used throughout the paper.

In the following, 
$\Omega$ is a nonempty, simply connected, bounded Lipschitz domain in $\mR^d$, $\omega$ is a compactly contained subdomain of $\Omega$, 
and
$\Gamma_0$ is a subset of $\partial \Omega$ with $\cH^{d-1}(\Gamma_0) > 0$,
where 
$\cH^{d-1}$ stands for the $(d-1)$-Hausdorff measure.
By Sobolev embedding theorem, \UUU given \EEE $y\in W^{1,q}(\Omega, \mR^d)$, $q>d$,
we \UUU can consider its \EEE continuous up to the boundary representative \OOO $\tilde{y}\in C^0(\bar{\Omega},\mR^d)$. \EEE
Therefore, the boundary condition
$y|_{\Gamma_0} = \textrm{id}$
\UUU is interpreted as \EEE $\tilde{y}(x) = x$
for every $x\in \Gamma_0$.

\OOO
Unless otherwise stated, throughout the paper
\UUU 
we will use the symbol $C$ to indicate any
generic positive constant, possibly depending on data, and changing
even within the same line. \EEE



\subsection{\EEE Sobolev spaces with \UUU vanishing trace}\label{sec:Sobolev}

\OOO 
In what follows,
\UUU Sobolev functions vanishing at the boundary of
the deformed set $y(\Omega)$ will turn out to be relevant. \EEE 
 The boundary of the set $y(\Omega)$ may, in fact, \UUU
show poor regularity. We hence need to introduce a characterization of
 Sobolev spaces with vanishing trace at the boundary of the \EEE set $y(\Omega)$.

\UUU Given a domain $D\UUU \subset \R^d$, different definitions of
spaces with vanishing trace at the boundary $\partial D$ can be
considered. One possibility is letting  $W^{1,2}_0(D)$ be \EEE the space defined as the closure of $C_0^{\infty}(D)$ in the $W^{1,2}$-norm.
\UUU An alternative is defining \EEE  $\accentset{\,\circ}{W}^{1,2}
(D)$ \UUU to \EEE be the set of functions in $W^{1,2}(\mR^d)$ 
that are equal to zero a.e.~in $\mR^d \setminus D$.
It follows directly from \UUU these  definitions \EEE that 
\begin{equation}\label{sobolev_inclusion}
    W^{1,2}_0(D) \subset \accentset{\,\circ}{W}^{1,2} (D).
\end{equation}
The opposite inclusion, and hence the equality of these two spaces,
holds \UUU for domains \EEE with $C^0$ boundaries. \UUU Note however
that such continuity is difficult to ascertain a priori, for \EEE even
the image of a smooth set via a homeomorphism might, in principle,
\UUU have no \EEE $C^0$ boundary. We refer
to~\cite{ChaWilHewMoi2017} for a detailed discussion \UUU of \EEE this topic.

\subsection{Smooth approximation of sets}\label{sec:smooth_approx}
\UUU The approximation theory from \cite{BalZar2017} entails that,  \EEE
for every open set $A\subset \mR^d$, $\bar A\neq \mR^d$, there exist a
constant $\ep_0>0$ and a collection of $C^{\infty}$-smooth open sets
$\{A_\ep\}_{0<\ep<\ep_0}$ approximating $A$ from inside \UUU in the
following sense \EEE
\begin{center}
    $\bigcup_{0<\ep<\ep_0} A_{\ep} = A$
    and
    $\overline{A_{\ep}} \subset A_{\ep'}$
    if $0<\ep'<\ep<\ep_0$.
\end{center}
These sets are classically \UUU defined by means of \EEE so-called 
{\it thinnings} of $A$, \UUU namely, \EEE
\begin{equation}
\label{eq:thinning}
\begin{aligned}
    A_{\ep}& :=\left\{x\in\mR^d: \widetilde{\dist}(x, \mR^d\setminus A)>\ep\right\},
\end{aligned}
\end{equation}
where $\widetilde{\dist}$ is a \UUU suitably \EEE regularized distance function, see~\cite[Remark 5.5]{BalZar2017}.

Let $K\subset \mR^d$ be a compact set with a nonempty interior. 
Consider the approximations $B_{\ep}$ of the open set $B:=\mR^d
\setminus K$ \UUU as in \eqref{eq:thinning}. \EEE
Clearly, 
\begin{equation}
\label{eq:thickeing}
    K^{\ep}\coloneqq\mR^d \setminus B_{\ep}=
    \left\{x\in\mR^d: \widetilde{\dist}(x, K)\leq\ep\right\}
\end{equation}
are compact sets with $C^{\infty}$-boundary, approximating $K$ from
outside \UUU in the following sense \EEE
\begin{center}
    $\bigcap_{0<\ep<\ep_0} K^{\ep} = K$
    and 
    ${K^{\ep'}} \subset {\rm int }\,K^{\ep}$
    if $0<\ep'<\ep<\ep_0$.
\end{center}

\section{\UUU Closure of admissible deformations}
\label{sec:adm_def}


We gather here some basic definitions and preliminaries from the
setting of quasiconformal analysis \UUU and comment on the closure of
the set $\cA$ of admissible deformations \eqref{cA} under uniform energy bounds.  \EEE

\begin{Definition}[\UUU Finite distortion]\label{def:FD}
	Let
	$f\colon\Omega\to \mR^d$
	be such that
	$f \in W^{1,1}_{\loc}(\Omega;\R^d)$
	and 
	$\det \nabla f (x) \geq 0$ for almost every $x\in \Omega$. 
	We say that $f$ is a 
	{\it mapping with finite distortion} if 
	for almost every 
	$x\in \Omega$
	it holds that 
	$\nabla f(x)=0$
    whenever
	$\det \nabla f (x) = 0$. 
    The function
    \begin{equation}\label{def:outer_inner_distortion}
    	K_{f,p}(x):=\begin{cases}\displaystyle\frac{|\nabla f(x)|}{\det \nabla f (x)^{1/p}}&\text{if }0 < \det \nabla f (x) < \infty,\\
    	0&\text{otherwise},\end{cases} 
    \end{equation}
    is called the \textit{outer distortion operator function} or \textit{outer $p$-distortion} of $f$
    at  %
    $x\in \Omega$.
\end{Definition}

The special \UUU case \EEE  $p=d$ \UUU in the definition above is one of the primary focuses of
quasiconformal analysis and  is \EEE analyzed in
\cite{IwaSve1993,VodGold1976}. If $K_{f,d}\in L^{\infty}(\Omega)$,
Definition~\ref{def:FD} \UUU corresponds to that of {\it quasiregular
  mappings}, also known as mappings with {\it bounded distortion}.
The general case in which $1\leq p<+\infty$, \UUU possibly with $p\not
= d$ is addressed \EEE in \cite{Vod2012} in connection with the study of the functional classes that preserve Sobolev mappings under change of variables.
We refer to~\cite{Resh1982,Rick1993} and to the monographs~\cite{HenKos2014,IwaMar2001} for overviews on the topics of mappings with bounded and finite distortion, respectively.

\begin{Definition}[\UUU Discrete and open maps]
We say that a continuous mapping 
$f\colon D \to D'$ 
is
{\it discrete} 
if 
$f^{-1}(y)$ 
is a discrete set for all $y \in D'$. \UUU If  $f(U)$ is open for
every open set $U\subset D$ we say that $f$ is {\it open}.  
\end{Definition}



The main properties of admissible deformations \UUU $\cA$ \EEE are collected in the next proposition.

\begin{Proposition}[\UUU Properties and closure of $\cA$]
\label{prop:properties_deformations_homeo} 
\OOO
    Let $y\colon \Omega \to \mR^d$ be such that $y|_{\Gamma_0}={\rm id}$ and
    \UUU
    $\cF_i(y)<\infty$, either for $i=1$ or $i=2$. \EEE
    Then, 
    \begin{enumerate}[label={$(\rm \roman*)$}]
        \item $y$ has finite distortion;
        \item $K_{y,d} \in L^{ds}(\Omega)$;
        \item $\det \nabla y >0$ a.e. in $\Omega$;  
\item $y$ is continuous, open, and discrete;
        \item $y$ satisfies the Lusin $\cN$ and $\cN^{-1}$ conditions.
    \end{enumerate}
    If, in addition, 
    $y$ is a weak limit of $ W^{1,q}$-homeomorphisms,
    then
    \begin{enumerate}[label={$(\rm \roman*)$}]\setcounter{enumi}{5}
        \item $y$ is injective a.e., both in \UUU the \EEE image and
          in \UUU its \EEE domain;
        \item $y$ is a homeomorphism.
    \end{enumerate}
    In particular, if \UUU $y^n\in \cA$ \EEE converge weakly to $y$ in
    $W^{1,q}(\Omega;\mR^d)$, then \UUU $y\in \cA$ as well. \EEE
\end{Proposition}
\begin{proof}
Properties (i) and (ii) follow immediately from \eqref{hp:growth}.
The proof of property (iii) for $W^{1,d}$-mappings with $K_{y,d} \in L^{ds}(\Omega)$, $s>d-1$, may be found in \cite[Theorem 1.1]{KosMal2003}.
\OOO
Continuity, openness, and discreteness for mappings with bounded distortion have been obtained in the seminal paper \cite{Resh1967-2}. 
Concerning mappings with finite distortion, it was shown in  \cite[Theorem 2.3]{VodGold1976}  that
$W^{1,d}_{\loc}$-mappings 
with finite distortion have continuous representatives. 
Now, since $y$ is non-constant by $y|_{\Gamma_0}={\rm id}$, then due to
\cite{ManVill1998}, (i) and (ii) imply that $y$ is open
and discrete, i.e., \UUU Property (iv) holds. \EEE
Property  (v) is a consequence of \cite[Proposition 2.4]{VodGold1976},
\OOO see \cite{MarMiz1973}.
\EEE
We refer the reader to \cite{HenKos2014}, where all the aforementioned
results and their consequences are \UUU discussed. \EEE

Property (vi) for limits of Sobolev homeomorphisms follows by \cite{BouHenMol2019}, see also \cite[Lemma 3.4]{MulSpe1995} and \cite[Theorem 10]{MolVod2020}.
Finally, the Lusin $\cN$-property, a.e.\ injectivity, and openness guarantee that $y$ is a homeomorphism, see, for example,  \cite[Lemma~3.3]{GraKruMaiSte2019}.
\end{proof}


\section{Capacity: main properties and upper semicontinuity}\label{sec:capacity}

Variational capacity is one of the main tools in nonlinear potential theory, see \cite{HeiKilMar2006}.
\UUU It delivers an \EEE essential understanding \UUU of \EEE the pointwise 
\OOO 
behavior
\EEE 
in the Sobolev setting, for it measures,  
roughly speaking,  the size of exceptional sets associated to Sobolev functions.
We refer the interested reader to \cite[Chapter 2]{HeiKilMar2006} and \cite{DalMaso} for a thorough discussion of the notion of capacity, as well as to \cite[Chapter 4]{HeiKilMar2006} for an overview on fine properties of Sobolev functions. We recall some basic properties below.

\begin{Definition}\label{def:capacity}
  Let \UUU $F,\,E\subset \R^d$ be compact with $E \subset D \subset
  \R^d$ 
  \OOO
  and
  \UUU 
  $D$ open. The
    \textit{capacity} of $F$ and the \textit{capacity of $E$ relative
    to  $D$} are \EEE defined by 
    \begin{align}
    \UUU  \capacity\,(F)&\UUU :=\inf\left\{\int\limits_{\R^d}|\nabla v(\xi)|^2\d\xi : v\in\cC_1(F)\right\},\label{eq:def-capacity0}\\
        \capacity\,(E;D)&:=\inf\left\{\int\limits_{D}|\nabla
      v(\xi)|^2\d\xi : v\in \UUU \cC_2\EEE (E;D)\right\},  \label{eq:def-capacity}   
    \end{align}   
    where
\begin{align*}
    \UUU  \cC_1(F)&\UUU:=\{\,v\in L^{2^*}(\R^d):\ \nabla v \in
  L^2(\R^d;\R^d), \ v\geq 1\text{ a.e.~in a
  neighb. of }F\},\\
 \UUU \cC_2(E;D)&\UUU:=\{\,v\in W^{1,2}_0(D):\ v\geq 1\text{ a.e.~in a neighb. of }E\}.
    \end{align*}
    Functions in \UUU $\cC_1(F)$ or $\cC_2(E;D)$ \EEE  are called \textit{\UUU capacity test \EEE functions}.
    We say that a property holds \textit{quasieverywhere}
    (\textit{q.e}), if it holds everywhere except from a set of
    zero capacity, and that a function is \emph{quasicontinuous} on
    $D$ if its discontinuity set in $D$ has zero capacity.
\end{Definition}

The next proposition collects some \UUU basic \EEE properties of Sobolev functions related to the notion of capacity (see \cite[Chapter 4]{HeiKilMar2006}).
\begin{Proposition}[Fine properties of Sobolev functions]
\label{prop:fine_properties}
    Let $D\subset \mathbb R^d$ be an open set. 
    \begin{enumerate}[label={$(\rm \roman*)$}]
        \item A function $v\in W^{1,2}_0(D)$ has a quasicontinuous representative $\tilde{v}$, uniquely 
        \OOO
        defined~q.e.
        \EEE
        \item Every strongly convergent sequence in $W^{1,2}(\mathbb R^d)$
            \UUU admits a 
            \OOO 
            q.e.\ convergent 
            \EEE subsequence in $\mathbb{R}^d$.
        \item 
            A function $u\in W^{1,2}(D)$ belongs to $W^{1,2}_0(D)$ if and only if its quasicontinuous representative  $\tilde{u}$ is the restriction to $D$ of a quasicontinuous map satisfying
$\tilde{u}=0$ q.e.\ on $\mathbb R^d \setminus D$.
    \end{enumerate}
\end{Proposition}

The notion of \UUU relative \EEE capacity can be equivalently
reformulated as follows.

\begin{Proposition}[\UUU Equivalent formulations]\label{prop:cap_qe}
Let \UUU $F, \, E \subset \R^d$ be compact with  $E\subset D\subset
\mR^d$ with $D$ open and bounded. \EEE
Then
\begin{align}
 \UUU \capacity\,(F)&= \UUU \inf\left\{\int\limits_{\R^d}|\nabla v(\xi)|^2\d\xi :
                 v\in\tilde \cC_1(F)\right\}, \label{eq:def-capacity-qe0}\\
    \capacity\,(E;D)&=\inf\left\{\int\limits_{D}|\nabla v(\xi)|^2\d\xi :  \UUU v\in\tilde \cC_2(E;D)\right\},   \label{eq:def-capacity-qe}
\end{align}  
where
\begin{align*}
  \UUU\tilde \cC_1(F)&\UUU:=\{v\in L^{2^*}(\R^d):\ \nabla v \in
                   L^2(\R^d;\R^d), \\
                 &\qquad \UUU \text{and its quasicontinuous representative
                   $\tilde v$ is such that $\tilde v\geq 1$ q.e. in $F$}
                   \},\\
 \UUU \tilde \cC_2(E;D)&\UUU:=\{v \in W^{1,2}_0(D)\\
  &\qquad \text{and its quasicontinuous representative
                   $\tilde v$ is such that $\tilde v\geq 1$ q.e. in $E$}
                   \}.
\end{align*}
\end{Proposition}

The behavior of capacity with respect to \UUU set inclusion \EEE is
encoded \UUU in \EEE the next proposition. 

\begin{Proposition}[Monotonicity properties of the capacity]
\label{prop:cap_prop} \EEE Let $F, \, F_k, \, E, \, E_k\subset \R^d$ be
compact and $D_k$ be bounded and open \EEE  for every $k\in\mN$. 
\UUU The \EEE following monotonicity properties \UUU hold: \EEE
    \begin{enumerate}[label={$(\rm \roman*)$}, ref=\ref{prop:cap_prop}\,(\roman*)]
        \item\label{prop:cap_monot_set}
        If \UUU $F_1 \subset F_2$ and \EEE $E_1 \subset E_2 \subset D$, then \UUU $\capacity(F_1) \leq
        \capacity(F_2)$ and \EEE $\capacity(E_1;D) \leq \capacity(E_2;D)$.
        \item\label{prop:cap_monot_domain} If $E \subset D_1 \subset D_2$, then $\capacity(E;D_2) \leq \capacity(E;D_1)$.
        \item\label{prop:inf-cap-outside} If $E = \bigcap_{k=1}^{+\infty} E_k$ with $E_{k+1} \subset E_{k} \subset D$
        and
        $E_k$ is compact for every $k\in \mN$, then 
        \begin{align*}
          \UUU  \capacity(E) &\UUU = \lim\limits_{k\to\infty}\capacity(E_k)= \inf\limits_{k\to\infty}\capacity(E_k),\\
            \capacity(E;D) &= \lim\limits_{k\to\infty}\capacity(E_k;D)= \inf\limits_{k\to\infty}\capacity(E_k;D).
        \end{align*}
        \item \label{prop:sup-cap-inside}
        If $E= \bigcup_{k=1}^{+\infty} E_k$ with $E_k \subset E_{k+1} \subset E \subset D$ for every $k\in \mathbb{N}$, then 
        \begin{align*}
           \capacity(E) &= \lim\limits_{k\to\infty}\capacity(E_k)= \sup\limits_{k\to\infty}\capacity(E_k),\\
            \capacity(E;D) &= \lim\limits_{k\to\infty}\capacity(E_k;D)= \sup\limits_{k\to\infty}\capacity(E_k;D).
            \end{align*}
          \item \label{prop:inf-cap-inside}
        If $D=\bigcup_{k=1}^{+\infty}D_k$ and $E$ is compact, with $E\subset D_k\subset\subset D_{k+1}\subset\subset D$ for every $k\in \mathbb{N}$, then
        \begin{equation*}
            \capacity (E;D)=
            \lim\limits_{k\to\infty}\capacity(E;D_k)=
            \inf_{k\in \mathbb{N}}\capacity(E;D_k).
        \end{equation*}
    \end{enumerate}
\end{Proposition}

\begin{proof}
The proof of Properties \ref{prop:cap_monot_set},
\ref{prop:inf-cap-outside} and \ref{prop:sup-cap-inside} can be found
in \cite[Theorem 2.2]{HeiKilMar2006} \UUU and \EEE \cite[Propositions 3.1, 4.1, 4.5]{DalMaso}. Note that the proof in \cite[Theorem 2.2]{HeiKilMar2006} is performed for the relative capacity but that the case of the capacity follows by the same argument.
 Property \ref{prop:cap_monot_domain} follows directly from the definition of capacity.


\UUU In order to \EEE prove \UUU Property \EEE \ref{prop:inf-cap-inside}, we first notice that
the inequality 
$\capacity (E;D)\leq \inf_{k\in \mathbb{N}}\capacity(E;D_k)$ follows
directly from \UUU Property \EEE \ref{prop:cap_monot_domain}. 
Thus, it suffices to prove the opposite inequality. For convenience of
the reader, we subdivide the proof into two steps.

{\sc Step 1}: assume that $\partial E$ is of class $C^{\infty}$. Let $\ep>0$, and consider a map $u\in W^{1,2}_0(D)$ with $u\geq 1$ quasi everywhere on $E$, and such that
\begin{equation}\capacity (E;D)\geq \int\limits_D |\nabla u|^2\,\d x-\ep.\label{dopo}
\end{equation}
\UUU By possibly \EEE replacing $u$ with $\bar{u}:=\min\{u,1\}$, we
can assume that $u\equiv 1$ quasi everywhere on $E$. Let now $\eta\in
C^{\infty}_c(D)$ be such that $\eta\equiv 1$ on $\bar{E}$. Note that
such cut-off function exists because $|\partial E|=0$. Set
$v:=u-\eta$. By definition, $v\equiv 0$ quasi everywhere on $E$. Thus,
we \UUU can \EEE find a sequence $\{v_n\}\subset C^\infty_c(D\setminus \bar{E})$ satisfying the following properties:
\begin{align}
&\label{eq:appr-vn}\|v_n-(u-\eta)\|_{H^1(D)}< \ep,\\
&\label{eq:supp-vn} {\rm supp}\, \UUU (v_n+\eta) \EEE \subset D_{k_n},
\end{align}
 for a suitable subsequence $\{D_{k_n}\}\subset \{D_k\}$. Consider now the maps $u_n:=v_n+\eta$. We have that
\begin{align}
&\label{eq:un-approx} \|u_n-u\|_{H^1(D)}=  \|v_n+\eta-u\|_{H^1(D)}<\ep,\\
&\label{eq:un-smooth} u_n\in C^\infty_c(D)\quad\text{for every }n\in \mathbb{N},\\
&\label{eq:un-1}
u_n\equiv 1\quad\text{ quasi everywhere on }E.
\end{align}
Therefore, $\{u_n\}\subset \mathcal{\tilde{A}}(E;D)$, where
$\mathcal{\tilde{A}}(E;D)$ is the class in \UUU Definition
\EEE\eqref{eq:def-capacity-qe}. Additionally, by \UUU \eqref{dopo} and
\EEE \eqref{eq:un-approx},
\begin{equation}\label{eq:step1-inf-cap}
\begin{aligned}
    \capacity (E;D)&\geq \int\limits_D |\nabla u|^2\,\d x-\ep\geq
    \int\limits_D |\nabla u_n|^2\,\d x-2\ep  \geq
    \int\limits_{D_{k_n}} |\nabla u_n|^2\,\d x-2\ep\\
    & \geq \capacity (E;D_{k_n})-2\ep.
\end{aligned}
\end{equation}
\UUU Due to the arbitrariness of $\ep>0$, this \EEE  yields \UUU
Property \EEE \ref{prop:inf-cap-inside} in the case of smooth sets~$E$.

\noindent {\sc Step 2}: Let now $E$ be an arbitrary compact subset of $D$. 
\UUU Arguing as in \EEE  Subsection~\ref{sec:smooth_approx}, we can
find a sequence of \UUU smooth \EEE compact sets $E_m$, approximating $E$ from outside.
In view of~\UUU Property \EEE \ref{prop:inf-cap-outside},
we have that
\begin{align*}
    \capacity(E;D)=\inf\Bigg\{\capacity(E_m;D)\ :\ &E_m\text{ is
  compact, }\partial E_m\text{ is }C^\infty,\\
  &E_{m+1}\subset {\rm int }\,E_{m},\,E=\bigcap_{m\in \mathbb{N}}E_m\Bigg\}.
\end{align*}
Fix $\ep>0$ and let $m(\ep)\in \mathbb{N}$ be such that
$$\capacity (E;D)\geq \capacity(E_{m(\ep)};D)-\ep.$$
By \eqref{eq:step1-inf-cap} and \UUU Property \EEE  \ref{prop:cap_monot_set}, we deduce the existence of an index $k(m,\ep)$ such that
\begin{align*}
    \capacity(E;D)\geq \capacity(E_{m(\ep)};D)-\ep 
    &\geq \capacity(E_{m(\ep)};D_{k(m,\ep)})-3\ep
    \geq \capacity(E;D_{k(m,\ep)})-3\ep \\
    &\geq \inf_{k\in \mathbb{N}} \capacity(E;D_k) - 3\ep.
\end{align*}
Given that $\ep$ is arbitrary,
this completes the proof of \UUU Property \EEE \ref{prop:inf-cap-inside}.
\end{proof}


\subsection{Upper semicontinuity of the capacity}
\UUU This section is devoted to the proof of the lower semicontinuity of
the capacitary terms in $\cF_1$ and $\cF_2$. This in particular rests
upon the upper semicontinuity of the capacity and the relative
capacity.   \EEE

\begin{Proposition}[\UUU Upper semicontinuity]\label{prop:limsup-cap}
Let 
$y\colon \UUU \bar \Omega  \EEE \to \mR^d$
and $\{y_n\}\subset C^0(\bar{\Omega};\mathbb{R}^d)$ be 
 homeomorphisms such that $y^n\to y$ strongly in
 $C^0(\bar{\Omega};\mathbb{R}^d)$.
Then,
\begin{align}
   &\UUU \limsup_{n\to +\infty}\capacity (y^n(\oomega))\leq \capacity (y(\oomega)),\label{neq:usc_capacity0}\\
    &\limsup_{n\to +\infty}\capacity (y^n(\oomega);y^n(\Omega))\leq \capacity (y(\oomega);y(\Omega)). \label{neq:usc_capacity}
\end{align}
\end{Proposition}

\begin{proof}
Since the maps $\{y^n\}$ are homeomorphisms, \UUU we have \EEE that $y^n(\oomega)$ is compact and $y^n(\Omega)$ is a domain for every $n\in \mN$.
Let
$\{\UUU E_m \EEE \}$ be a sequence of $C^\infty$ compact sets,
approximating $y(\oomega)$ from outside (see  Figure \ref{pic:1} below). \UUU The existence of such
approximating sets is discussed in Subsection
\ref{sec:smooth_approx}.  By the uniform convergence of the sequence $\{y^n\}$ we deduce that
$ y(\oomega)\cup y^n(\oomega)\subset {\rm int}\,\UUU E_m\EEE$  for $n\in \mathbb{N}$ \UUU
large  enough, and for every $m\in \mN$. Hence, Property
~\ref{prop:cap_monot_set} entails that
$$\limsup_{n\to +\infty}\capacity (y^n(\oomega)) \leq
\capacity(E_m)\quad \forall m \in \mathbb{N}.$$
By taking $m \to +\infty$ an using
Property~\ref{prop:cap_monot_domain} we get \eqref{neq:usc_capacity0}.

Let now \EEE $\{D_\ell\}$ be a sequence of $C^\infty$ open sets, approximating
$y(\Omega)$ from inside 
(see \UUU again \EEE Figure \ref{pic:1}). Again, the \UUU existence of such
approximating sets is discussed in Subsection \ref{sec:smooth_approx}. \EEE
By the uniform convergence \UUU we have that \EEE
$D_\ell\subset y(\Omega)\cap y^n(\Omega)$ for $n\in \mathbb{N}$ \UUU
large \EEE enough. 

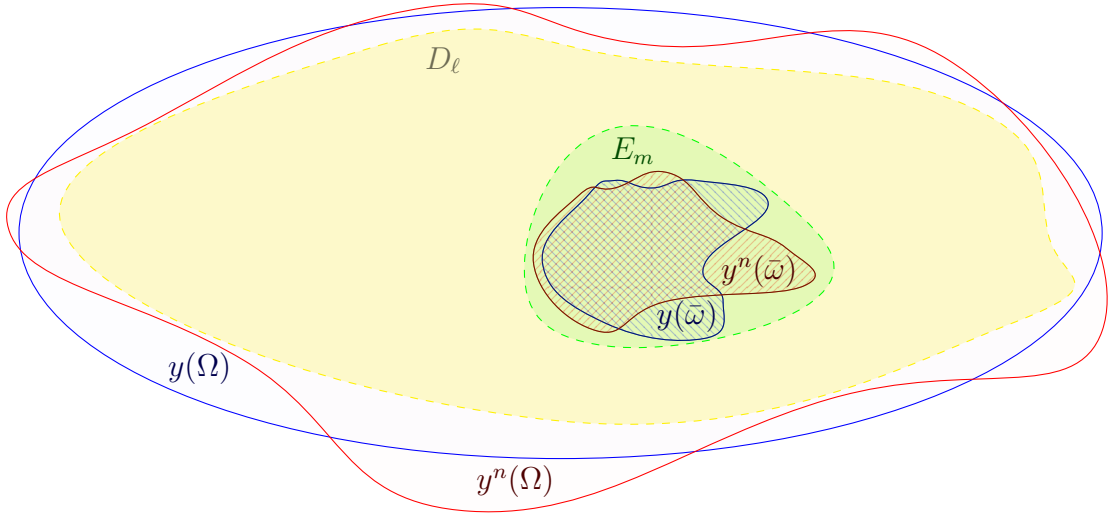
\begin{figure}[h]
\centering
\begin{tikzpicture}[xscale=0.12,yscale=0.1]
    \draw [blue] (0,0) ellipse (60 and 30);
    \fill [blue,opacity=0.01] (0,0) ellipse (60 and 30);
    \draw [red] plot [smooth cycle,tension=0.8] coordinates {(-61,0) (-45,15) (-15,30) (10,25) (43,23) (60,-13) (30,-22) (-10,-37) (-35,-17)};
    \fill [red, opacity=0.01] plot [smooth cycle,tension=0.8] coordinates {(-61,0) (-45,15) (-15,30) (10,25) (43,23) (60,-13) (30,-22) (-10,-37) (-35,-17)};
    \draw [yellow, dashed] plot [smooth cycle,tension=0.8] coordinates {(-55,0) (-20,25) (10,23) (45,17) (54,0) (50,-12) (0,-25)};
    \fill [yellow, opacity=0.2] plot [smooth cycle,tension=0.8] coordinates {(-55,0) (-20,25) (10,23) (45,17) (54,0) (50,-12) (0,-25)};
    
    \node[yellow!30!black] at (-13,23) {$D_\ell$};
    \node[blue!30!black] at (-40,-18) {$y(\Omega)$};
    \node[red!30!black] at (-5,-33) {$y^n(\Omega)$};
    
    \draw [blue] plot [smooth cycle,tension=0.8] coordinates {(-2,-4) (3,5) (6,7) (10,6) (15,7) (23,4) (16,-4) (18,-9) (16, -14) (5,-12)};
    \draw[pattern=north west lines, pattern color=blue, opacity=0.5] plot [smooth cycle,tension=0.8] coordinates {(-2,-4) (3,5) (6,7) (10,6) (15,7) (23,4) (16,-4) (18,-9) (16, -14) (5,-12)};
    \draw [red] plot [smooth cycle,tension=0.8] coordinates {(-3,-4) (2,5) (6,6) (12,8) (18,2) (25,-2) (27,-7) (12, -9) (4,-13)};
    \draw [pattern=north east lines, pattern color=red, opacity=0.5] plot [smooth cycle,tension=0.8] coordinates {(-3,-4) (2,5) (6,6) (12,8) (18,2) (25,-2) (27,-7) (12, -9) (4,-13)};
    \draw [green, dashed] plot [smooth cycle,tension=0.8] coordinates { (-4,-4) (6,14) (25,4) (28,-9) (5,-15)};
    \fill [green, opacity=0.1] plot [smooth cycle,tension=0.8] coordinates { (-4,-4) (6,14) (25,4) (28,-9) (5,-15)};
    
    \node[green!30!black] at (8,11) {$E_m$};
    \node[blue!30!black] at (14,-11) {$y(\bar\omega)$};
    \node[red!30!black] at (22,-5) {$y^n(\bar\omega)$};
\end{tikzpicture}
\caption{Sets $D_\ell$ and $\UUU E_m$.}
\label{pic:1}
\end{figure}


Then,   \UUU from Properties \EEE ~\ref{prop:cap_monot_set}
and~\ref{prop:cap_monot_domain} \UUU we deduce that \EEE
$$\capacity (y^n(\oomega);y^n(\Omega))\leq \capacity (\UUU E_m
\EEE;y^n(\Omega))\leq \capacity (\UUU E_m \EEE ;D_\ell)$$ for $n$ \UUU
large \EEE enough, and for every $m$ and $\ell$. In particular,
\begin{align*}
   & \limsup_{n\to+\infty}\capacity (y^n(\oomega);y^n(\Omega))\leq
     \inf_{\ell}\inf_{m}\capacity (\UUU E_m\EEE;D_\ell)\\
   &\quad=\inf_{\ell}\capacity (y(\oomega);D_\ell)=\capacity (y(\oomega);y(\Omega)),
    \end{align*}
where the second-to-last equality follows by \UUU Property~\ref{prop:inf-cap-outside}, 
\EEE and the last one by~ \UUU Property~\ref{prop:inf-cap-inside}.
\EEE
\end{proof}



\section{Proof of Theorem~\ref{thm:main}}
\label{sec:proof-main}
This section is devoted to the proof of our main result, \UUU
Theorem~\ref{thm:main}. This follows from an application of the  Direct
Method.  \EEE

Let \UUU $\{y^n_1\}_{n\in\mN}$, $\{y^n_2\}_{n\in\mN} \OOO \subset \cA \EEE$ be
minimizing sequences for the functionals $\cF_1$ and $\cF_2$,
respectively. 
The compactness of the \OOO sequences 
$\{y^n_1\}_{n\in\mN} $ and $ \{y^n_2\}_{n\in\mN}$ \EEE  follows directly from \eqref{hp:growth}, and from the observation that
\begin{align*}
   &\UUU \liminf_{n\to +\infty} \cF_1(y^n_1)
    \leq \cF_1(\textrm{id})= \int\limits_\Omega
     W(x,\textrm{Id})\, \d x  +\frac{Q^2}{2\capacity(\oomega)}<+\infty,
  \\
  &\UUU \liminf_{n\to +\infty} \cF_2(y^n_2) \leq \cF_2(\textrm{id})=
  \int\limits_\Omega
     W(x,\textrm{Id})\, \d x+\frac{Q^2}{2\capacity(\oomega;\Omega)}<+\infty
    .   
\end{align*}
Hence, there \MMM exist \OOO  $y_1$, $y_2 \in W^{1,q}(\Omega;\mR^d)$ \EEE such that, up to extracting a not relabeled subsequence, 
\begin{equation}
\label{weak_conv}
    \UUU y^n_i\wk y_i \text{ weakly in } W^{1,q}(\Omega;\mR^d)\quad \UUU
    \text{for} \  i=1,\,2. \EEE
\end{equation}
%
%
%
%
%
Then, 
\begin{equation}
\label{neq:lsc_gradient}
    \int\limits_{\Omega} |\nabla \UUU y_i\EEE |^q \, \d x \leq
    \liminf_{n\to +\infty}\int\limits_\Omega |\nabla \UUU
    y^n_i\EEE|^q\,\d x \quad \UUU
    \text{for} \  i=1,\,2. \EEE
\end{equation}
Moreover,  by \eqref{weak_conv}, it follows that
\begin{equation}
\label{eq:wk-det}    
\det \nabla \UUU y^n_i \EEE  \wk \det \nabla \UUU y_i \EEE
\quad\text{weakly in }L^{\UUU q/d \EEE }(\Omega) \quad \text{for} \  i=1,\,2,
\end{equation} 
and so $\det \nabla \UUU y_i \EEE  \geq 0$ almost everywhere in
$\Omega$, since $\det \nabla \UUU y^n_i \EEE \geq 0$ almost everywhere in $\Omega$.
Indeed, by \eqref{eq:wk-det} and Mazur's lemma, we find linear
combinations $\UUU d^n_i \EEE $ \UUU of $\det \nabla \UUU y^n_i \EEE $
\EEE such that $\UUU d^n_i  \EEE\to \det \nabla \UUU y_i \EEE$
strongly in $L^{\UUU q/d \EEE}(\Omega)$, and $\UUU d^n_i \EEE \geq 0$
a.e. This yields that $\det \nabla \UUU  y_i \EEE \geq 0$ almost everywhere in $\Omega$.

Let now $\UUU K_i \EEE $ be a weak limit of $K_{\UUU y^n_i \EEE,d}$ in $L^{ds}(\Omega)$. 
Then by \cite{GehIwa1999} (see also \cite[Theorem
8.10.1]{IwaMar2001} \OOO and \cite{VodMol2016}), \UUU $y_i $ \EEE has finite distortion and
\begin{equation}
\label{neq:lsc_distortion}
    \int\limits_\Omega \left(K_{\UUU y_i \EEE,d}(x)\right)^{ds}\,\d x
    \leq \liminf_{n\to +\infty}\int\limits_\Omega \left(K_{\UUU y^n_i
        \EEE ,d}(x)\right)^{ds}\,\d x.
\end{equation}

Note also that neither \UUU $\{y^n_i\}$ nor $y_i$ \EEE can be constant owing to the boundary conditions
$\UUU y_i|_{\Gamma_0}=y^n_i \EEE |_{\Gamma_0} = \textrm{id}$ (see Subsection \ref{subs:geo}).
Therefore, \UUU $y_i \in \cA$ \EEE and satisfies \UUU Properties~(i)--(vii)
\EEE of Proposition~\ref{prop:properties_deformations_homeo}.

In view of the Sobolev embedding theorem and \UUU the weak convergence
\EEE~\eqref{weak_conv}, we may assume $y^n_i \to y_i$ \EEE strongly in \OOO $C^0(\bar{\Omega}, \mR^n)$. \EEE
Combining~\eqref{hp:growth}, \eqref{neq:lsc_gradient},
\eqref{neq:lsc_distortion}, \UUU and~\eqref{neq:usc_capacity0}, \EEE we obtain
 
\begin{align*}
 \UUU \cF_1(y_1)& \UUU =
    \int\limits_\Omega W(x,\nabla y_1(x)) \,\d x
              +\frac{Q^2}{2\capacity\,(y_1(\bar{\omega}))} \\
    & \UUU \leq \liminf_{n\to +\infty}\int\limits_\Omega W(x,\nabla y^n_1(x))\,\d x
    +
      \frac{Q^2}{2\limsup\limits_{n\to+\infty}\capacity\,(y^n_1(\bar{\omega}) )}
    \\
    &\UUU  = \liminf_{n\to +\infty}\left(\int\limits_\Omega W(x,\nabla y^n_1(x))\,\d x
    + \frac{Q^2}{2\capacity\,(y^n_2(\bar{\omega}))}\right)
    = \liminf_{n\to +\infty} \cF_1(y^n_1) = \inf_{\cA}\cF_1
\end{align*}
\UUU and, analogously,
$$\cF_2(y_2)\leq  \liminf_{n\to +\infty} \cF_2(y^n_2) = \inf_{\cA}\cF_2,$$
so that the statement of Theorem~\ref{thm:main} follows. \EEE   

\MMM \section{Continuity of the capacity and geometry of deformed sets}
\label{sec:capacity-new}

In this section we complete our study of continuity properties of the capacity by investigating its lower semicontinuity
under some additional requirements on the geometry of
the deformed configuration. \UUU Note that the proof of Theorem
\ref{thm:main} does not rely on such lower semicontinuity.  \EEE

\subsection{Lower semicontinuity of the capacity}
\label{subs:lsc-cap}

\begin{Proposition}[\UUU Lower semicontinuity]\label{prop:lsc_cap}
    Let 
    $y\colon \UUU \bar \Omega\EEE \to \mR^d$
    and $\{y^n\}\subset C^0(\bar{\Omega};\mR^d)$
    be homeomorphisms such that $y^n\to y$ strongly in
    $C^0(\bar{\Omega};\mathbb{R}^d)$. \UUU Then,
 \begin{align}
    \label{neq:lsc_capacity0}\UUU\capacity(y(\bar{\omega})) &\UUU\leq
                                                              \liminf\limits_{n\to+\infty}
                                                              \capacity(y^n(\bar{\omega}))
                                                              .
    \end{align}
    
    Suppose \UUU additionally that \EEE  $y$
    is such that 
    $W^{1,2}_0(y(\Omega)) = \accentset{\,\circ}{W}^{1,2} (y(\Omega))$. 
    Then,  
    \begin{align}
     \label{neq:lsc_capacity}
        \capacity(y(\bar{\omega}); y(\Omega)) &\leq \liminf\limits_{n\to+\infty} \capacity(y^n(\bar{\omega}); y^n(\Omega)).
    \end{align}
\end{Proposition}

\begin{proof}
  \UUU Let us start by checking \eqref{neq:lsc_capacity}. \EEE
Without loss of generality, we may assume that 
that right-hand-side of \eqref{neq:lsc_capacity} is finite.

By Proposition~\ref{prop:cap_qe}, for every $n\in\N$, we find 
$v^n\in W^{1,2}_0(y^n(\Omega))$ with
quasicontinuous representative
$\tilde{v}^n$, such that $\tilde{v}^n \geq 1$ q.e.~on $y^n(\oomega)$, and
$$
    \int\limits_{y^n(\Omega
    )} |\nabla v^n (\xi)|^2 \,\d\xi \leq \capacity(y^n(\oomega); y^n(\Omega)) + \frac{1}{n}.
$$
Let $v_{ext}^n$ be an extension of $v^n$ by $0$ outside $y^n(\Omega)$, 
i.e.,
$v_{ext}^n \in W^{1,2}_0(\mathbb{R}^d)$
and 
$\tilde{v}_{ext}^n = 0$ a.e.\ on $\mathbb{R}^d \setminus y^n(\Omega)$.
Then, for all $n\geq n_0$ it holds that
\begin{equation*}
    \int\limits_{\mathbb{R}^d} |\nabla v_{ext}^n (\xi)|^2 \,\d\xi
    = \int\limits_{y^n(\Omega)} |\nabla v^n (\xi)|^2 \,\d\xi.
\end{equation*}

Therefore 
$\{v^n_{ext}\}_{n\in\mN}\subset W^{1,2}_0(\mathbb{R}^d)$ is
bounded. \UUU Hence,  \EEE
there exists 
$v\in W^{1,2}_0(\mathbb{R}^d)$ 
such that, up to subsequence, 
$v^n_{ext} \wk v$ weakly in $W^{1,2}_0(\mathbb{R}^d)$.

We proceed by showing that $v\in\mathcal{\tilde{\UUU C}}_2(y(\bar{\omega}),y(\Omega))$.
By Mazur's lemma (see, e.g., \cite[p.~6]{EkeTem1976}) we \UUU find \EEE a sequence
$\{u^n\} \subset W^{1,2}_0(\mathbb{R}^d)$ such that 
$u^n \to v$ strongly in $W^{1,2}_0(\mathbb{R}^d)$ \UUU with the
property that $u^n$ is a convex combination of $\{v^n_{ext},
v^{n+1}_{ext}, \dots \}$. \EEE
In particular, denoting by $Y^n(\oomega)$ the set
$Y^n(\oomega):=\bigcap\limits_{k=n}\limits^{\infty} y^k(\oomega)$, we have that
the quasicontinuous representatives $\{\tilde{u}^n\}$ associated to $\{u^n\}$ satisfy
$\tilde{u}^n \geq 1$ q.e.~on
$\bigcap\limits_{k=n}\limits^{N_n} y^k(\oomega) \supset Y^n(\oomega)$
and
$\tilde{u}^n = 0$ q.e.~on $\mR^d \setminus \bigcup\limits_{k=n}\limits^{N_n} y^k(\Omega)$, for a suitable integer $N_n\geq n$.

In view of \UUU Property (ii) of \EEE
Proposition~\ref{prop:fine_properties}  we infer that, up to subsequences, 
\begin{equation}
    u^n  \to v \quad \text{q.e.~on } \mathbb{R}^d. \label{conv:un}
    \end{equation}
Additionally,
\begin{equation*}
	|\chi_{Y^n(\bar{\omega})} u^n  - \chi_{y(\bar{\omega})} v |
	\leq 
	|\chi_{Y^n(\bar{\omega})} (u^n - v)| +
    |(\chi_{Y^n(\bar{\omega})} - \chi_{y(\bar{\omega})}) v|.
\end{equation*}
The first term converges to \UUU $0$ as $n \to \infty$ \EEE due to the fact that
$\|\chi_{Y^n(\bar{\omega})}\|_{L^{\infty}(\mathbb R^d)} \leq 1$,
and by \eqref{conv:un}. The second term is infinitesimal owing to the uniform convergence of $\{y^n\}$. Thus,
\begin{equation}
    \chi_{Y^n(\bar{\omega})} u^n  \to \chi_{y(\bar{\omega})} v \quad \text{q.e.~on } \mathbb{R}^d, \label{conv:un_omega}
\end{equation}
and hence, $v=1$ q.e. on $y(\bar{\omega})$.

\UUU We now \EEE show that $v=0$ a.e.~in $\mR^d \setminus y(\Omega)$.
For \UUU any \EEE bounded measurable set $F\subset \mR^d \setminus y(\Omega)$, we have
\begin{equation*}
    \left| \int\limits_{F} v(\xi) \,\d\xi \right| = \lim\limits_{n\to\infty}\left| \int\limits_{F} v^n_{ext}(\xi) \,\d\xi \right|
    \leq \lim\limits_{n\to\infty} \int\limits_{y^n(\Omega)\setminus y(\Omega)} |v^n_{ext}(\xi)| \,\d\xi =0.
\end{equation*}
The last equality follows from the equiintegrability of $\{v^n_{ext}\}_{n\in\mN}$, 
as well as from the fact that 
$|y^n(\Omega)\setminus y(\Omega)| \to 0$ 
(due to the uniform convergence of $\{y^n\}$).
Therefore, we conclude that $v\in \accentset{\,\circ}{W}^{1,2} (y(\Omega))$.
Since
$\accentset{\,\circ}{W}^{1,2} (y(\Omega)) = W^{1,2}_0(y(\Omega))$, this yields that
$v\in\mathcal{\tilde{A}}(y(\bar{\omega}),y(\Omega))$.

The lower semicontinuity of the capacity follows \UUU then \EEE
from the chain of inequalities 
\begin{align*}
    \capacity(y(\bar{\omega}); y(\Omega)) &\leq
    \int\limits_{y(\Omega)} |\nabla v (\xi)|^2 \,\d\xi  \leq \int\limits_{\mathbb{R}^d} |\nabla v (\xi)|^2 \,\d\xi 
    \leq \liminf_{n\to\infty} \int\limits_{\mathbb{R}^d} |\nabla v^n_{ext} (\xi)|^2 \,\d\xi 
    \\
    & = \liminf_{n\to\infty} \int\limits_{y^n(\Omega)} |\nabla v^n (\xi)|^2 \,\d\xi 
    \leq \liminf_{n\to\infty} \capacity(y^n(\bar\omega);
      y^n(\Omega)). 
\end{align*}

\UUU 
The proof of 
\eqref{neq:lsc_capacity0} follows the same lines as above, with the
simplification of not requiring extensions. In particular, we can find
a sequence $\{v^n\}_{n\in \mathbb{N}} \subset \tilde
\cC_1(y^n(\oomega))$ with
$$ \int\limits_{\R^d} |v^n(\xi)|\, \d \xi \leq \capacity
(y^n(\oomega)) +\frac{1}{n}$$
such that, at least for a not relabeled subsequence, $\nabla v^n \wk
\nabla v$ weakly in $L^2(\R^d,\R^d)$ with $v \in  \tilde
\cC_1(y(\oomega))$. We hence have that
$$ \capacity(y(\bar{\omega}))\leq
    \int\limits_{\R^d} |\nabla v (\xi)|^2 \,\d\xi  \leq
    \liminf_{n\to\infty} \int\limits_{\mathbb{R}^d} |\nabla v^n  (\xi)|^2 \,\d\xi 
     = \liminf_{n\to\infty}  \capacity(y^n(\bar\omega)).\qedhere
$$ \EEE
\end{proof}

\subsection{\UUU On the regularity of the  deformed $y(\Omega)$}\label{subs:geo}



\UUU In case of the relative capacity $\capacity(y(\oomega),
y(\Omega))$, the lower semicontinuity result from Proposition
\ref{prop:lsc_cap} is conditional to the fact that
$W^{1,2}_0(y(\Omega)) = \accentset{\,\circ}{W}^{1,2} (y(\Omega))$, for
this is needed in order to have that \EEE limiting maps $v$ satisfy 
$v \in H^1_0(y(\Omega))$. 


As \UUU already mentioned  \EEE in Subsection \ref{sec:Sobolev}, the
two spaces \UUU above \EEE can be identified \UUU in case the boundary
of the deformed set $y(\Omega)$ is $C^0$.  
On the other end, \EEE even if $y$ is a homeomorphism arising as
uniform limit of homeomorphisms $\{y^n\}$ for which $y^n(\Omega)$ is a
$C^0$ domain for every $n$, it is a priori not guaranteed that the
boundary of $y(\Omega)$ will \UUU show such regularity. \EEE

An explicit counterexample is provided by the Koch snowflake $\mathfrak{X}$. This set does not have a $C^0$-boundary and can be realized as image of a ball $B(0,1)$ under a quasiconformal map $y\colon \mR^d \to \mR^d$. In dimension $d=2$, this follows from Ahlfors' three point condition \cite{Ahl1963}, see also \cite[Theorem 2.7]{JerKen1982}; 
the case $d=3$ is studied in \cite{May2010}. \UUU On the other hand,
\EEE  $y$ is the uniform limit of mappings $\{y^n\}$, such that $y^n(\Omega)$ is a $C^0$-domain $\mathfrak{X}_n$ (polyhedron) for every $n\in \mN$, cf. Figure \ref{pic:2} below.

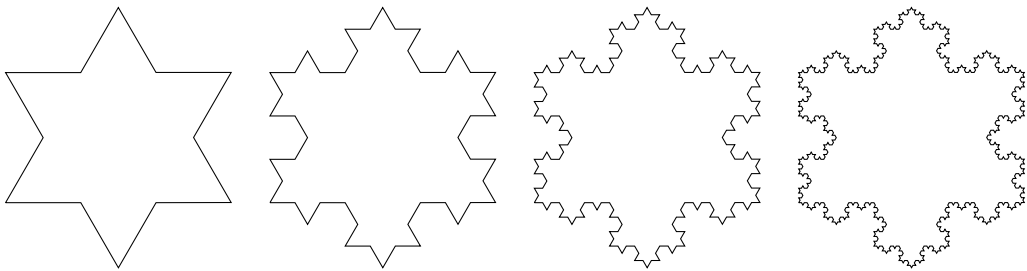
\begin{figure}[h]
\center
\begin{tikzpicture}[decoration=Koch snowflake]
   	\draw [xshift=100] decorate{ 
        (0,0) -- ++(60:3)  -- ++(300:3) -- ++(180:3)};
    \draw [xshift=200] decorate{ decorate{ 
        (0,0) -- ++(60:3)  -- ++(300:3) -- ++(180:3)}};
    \draw [xshift=300] decorate{ decorate{ decorate{ 
        (0,0) -- ++(60:3)  -- ++(300:3) -- ++(180:3)}}};
       \draw [xshift=400] decorate{ decorate{ decorate{ decorate{
        (0,0) -- ++(60:3)  -- ++(300:3) -- ++(180:3)}}}};
\end{tikzpicture}
\caption{Approximation of the Koch snowflake $\mathfrak{X}_j$, $j=1,2,3,4$}
\label{pic:2}
\end{figure}

\UUU Note however that \EEE the lack of regularity for the boundary of
the deformed set $y(\Omega)$ may be overcome \UUU by imposing \EEE
additional constraints on the approximating deformations $\{y^n\}$,
for \UUU instance,   by requiring \EEE that the sets $y^n(\Omega)$ are Lipschitz with the same constant $L$ for all $n\in \mN$.
A slightly weaker assumption would be to impose that all sets $y^n(\Omega)$ are uniformly regular:
a bounded set $D\subset \mR^d$ is called \textit{regular} if there are positive constants $b$ and $r_0$ such that for all $z\in\partial D$ and all $0<r\leq r_0$ there holds
$$
    |B(z,r) \cap (\mR^d \setminus D)| \geq b |B(z,r)|.
$$
In other words, a set is regular if \UUU the density of its \EEE
complement is  \UUU large \EEE enough. This holds, for example, for sets satisfying an outer cone condition. 
Recall the following characterization of Sobolev functions with zero traces,
taken from \cite[Theorems 4.1--4.2]{EdmNek2017}.
The first part of the statement below may be found in \cite[Theorem V.3.4 and Remark V.3.5]{EdmEva1987}. 
For the second part, under a cone property of $D$ a proof can be found in \cite{KadKuf1999} and \cite[Theorem X.6.7 and Remark X.6.8]{EdmEva1987},
we also refer to \cite{KinMar1997} for the case of weaker integrability assumptions.

\begin{Lemma}\label{lem:zero_trace_characterisation}
    Let $D$ be a bounded domain, $1 \leq p < \infty$, and for every $\xi\in \mR^d$, let
    $d(\xi):=\dist(\xi, \mR^d \setminus D)$. 
    If $u/d\in L^p(D)$ and $u\in W^{1,p}(D)$, then 
    $u\in W^{1,p}_0(D)$.
    
    If, instead, $D$ is a bounded regular domain and $1 <p<\infty$, then 
    $u\in W^{1,p}_0(D)$ if and only if $u/d \in L^1(D)$ and $u\in W^{1,p}(D)$.
\end{Lemma}

The next lemma follows from \cite[Theorem 3.3]{HarHasKos2005} and \cite[Proposition 1]{Haj1999}.
\begin{Lemma}\label{lem:ud}
    Let $D$ be a bounded regular domain in $\mR^d$ and $1<p<\infty$.
    Then there exists a constant $C$, depending only on $p$, $d$, $b$ and $r_0$, such that the inequality 
    \begin{equation}\label{neq:u/d}
        \left\|\frac{u}{d}\right\|_{L^p(D)} \leq C \|\nabla u\|_{L^p(D)}
    \end{equation}
    holds for all 
    $u\in W^{1,p}_0(D)$.
\end{Lemma}

In the next lemma, we show that the condition of regularity of deformed domains is closed under weak Sobolev convergence.

\begin{Lemma}\label{lem:dom_regularity}
    Let 
    $y$, $y^n\in W^{1,q}(\Omega;\mR^d)$, $q>d$,
    be homeomorphisms such that $y^n\wk y$ weakly in 
    $W^{1,q}(\Omega;\mR^d)$,
    and $y^n(\Omega)$ is a regular domain with constants $b$ and $r_0$ for every $n\in \mN$.
    Then $y(\Omega)$ is also a regular domain.
\end{Lemma}
\begin{proof}
    From classical set theory there holds 
    $A \cap B \supset (C\cap D) \setminus ((C\setminus A) \cup (D \setminus B))$, and
    hence, 
    $|A \cap B| \geq |C\cap D| - |C\setminus A| - |D \setminus B|$ for
    every collection of sets $A,\,B,\,C$, and $D$.
    By Sobolev embedding theorems we can assume that 
    $y^n \to y$ strongly in \OOO $C^0(\bar{\Omega}, \mR^d)$.\EEE
    Let $z \in \partial y(\Omega)$, and let $\{z_n\}$ be a sequence of points such that $z^n \in \partial y^n(\Omega)$ for every $n\in \mN$, and $z^n \to z$. Then, from the uniform regularity of the sets $y^n(\Omega)$ we find
    \begin{equation*}
    \begin{aligned}
        |B(z,r) \cap (\mR^d \setminus y(\Omega))| 
        & \geq |B(z^n,r) \cap (\mR^d \setminus y^n(\Omega))| -
        |B(z^n,r)\setminus B(z,r)| - |y(\Omega)\setminus y^n(\Omega)| \\
        & \geq b |B(z^n,r)| -
        |B(z^n,r)\setminus B(z,r)| - |y(\Omega)\setminus y^n(\Omega)|.
    \end{aligned}
    \end{equation*}
    Now fix $\ep>0$. For $n_0\in \mN$ \UUU large \EEE enough there holds $|B(z^n,r)\setminus B(z,r)| \leq \ep$ 
    and
    $|y(\Omega))\setminus y^n(\Omega))|\leq |\partial y(\Omega)| + \ep$
    for every $n \geq n_0$.
    Hence, 
    \begin{equation*}
        |B(z,r) \cap (\mR^d \setminus y(\Omega))| 
        \geq b |B(z^n,r)| -
        |\partial y(\Omega)| - 2\ep,
    \end{equation*}
    for every $n \geq n_0$.
    Since $y\in W^{1,q}(\Omega;\mR^d)$ and 
    $\Omega$ is a bounded Lipschitz domain,
    there exists an extension
    $\tilde{y}\in W^{1,q}(\tilde{\Omega};\mR^d)$, \UUU where $\tilde
    \Omega\supset \bar \Omega$ is a domain with smooth boundary, \EEE
    such that, \UUU by \EEE  identifying the maps with their continuous representative, \OOO
    $\tilde{y}|_{\bar{\Omega}} = y|_{\bar{\Omega}}$. \EEE
    Since $y$ is a homeomorphism, arguing as in  \cite[Lemma 3.1]{kruzik.melching.stefanelli}
    we have that
    $|\partial y(\Omega)|=|y(\partial \Omega)| = |\tilde{y}(\partial \Omega)|=0$ as $|\partial \Omega|=0$ and $\tilde{y}$ satisfies the Lusin $\cN$-condition, see \cite{MarMiz1973}.
    Owing to the arbitrariness of $\ep$, passing to the limit as $n\to +\infty$ we obtain
     \begin{equation*}
        |B(z,r) \cap (\mR^d \setminus y(\Omega))| 
        \geq b |B(z,r)|,
    \end{equation*}
    which in turn yields that $y(\Omega)$ is regular with constants $b$ and $r_0$.
\end{proof}

Combining Proposition~\ref{prop:lsc_cap} and
Lemmas~\ref{lem:zero_trace_characterisation}--\ref{lem:dom_regularity},
we \UUU are in the position of presenting \EEE a lower semicontinuity
result for the capacity \UUU under no a priori condition on the
regularity of the boundary \EEE of $y(\Omega)$ but under uniform regularity of $\{y^n(\Omega)\}$.

\begin{Proposition}
    Let 
    $y$, $y^n\in W^{1,q}(\Omega;\mR^d)$, $q>d$,
    be homeomorphisms such that $y^n\wk y$ weakly in 
    $W^{1,q}(\Omega;\mR^d)$,
    and $y^n(\Omega)$ is a regular domain with constants $b$ and $r_0$ for every $n\in \mN$.
    Then \eqref{neq:lsc_capacity} holds.
\end{Proposition}
\begin{proof}
    In view of Lemma \ref{lem:zero_trace_characterisation}, following the proof of Proposition~\ref{prop:lsc_cap}, it is enough to show that 
    $v/d$ is bounded in $L^2 (y(\Omega))$.
    First, from Lemma~\ref{lem:ud}, we obtain the uniform bound
    \begin{equation}\label{neq:vdbound}
        \int\limits_{y^n(\Omega)}\left|\frac{v^n(\xi)}{d^n(\xi)}\right|^2 \, \d \xi  
        \leq C\int\limits_{y^n(\Omega)}\left|\nabla v^n\right|^2 \d \xi \leq
        \UUU C\EEE,
    \end{equation}
    for a constant $C$ independent of $n$.
    
    Moreover, up to subsequences, 
    \begin{equation}\label{conv:vd}
    \frac{v^n(\xi)}{d^n(\xi)} \chi_{y^n(\Omega)}(\xi) \to \frac{v(\xi)}{d(\xi)} \chi_{y(\Omega)}(\xi) \quad \text{a.e.\ in } \mR^d.
    \end{equation}
    Indeed, 
    the almost everywhere convergence of $\{v^n\}$ to $v$ follows from Sobolev embeddings, whereas the pointwise convergence of $\{d^n\}$ to $d$ results from the uniform convergence of $y^n$ to $y$.
    Convergence \eqref{conv:vd} follows then directly \UUU for \EEE
    $\xi\in y(\Omega)$, as $d^n(\xi)$, $d(\xi) > \alpha >0$ for $n$
    \UUU large \EEE enough. Analogously, if $\xi\in \mR^d \setminus
    \overline{y(\Omega)}$, then $\xi\in \mR^d \setminus
    \overline{y^n(\Omega)}$ for $n$ big enough, and so both sides of
    \eqref{conv:vd} are equal to $0$. Arguing as in the proof of
    Lemma~\ref{lem:dom_regularity}, we \UUU find \EEE that $\partial
    y(\Omega)$ has measure zero, \UUU which completes \EEE the proof of \eqref{conv:vd}.
    
    From the pointwise convergence \eqref{conv:vd} and from the boundedness in~\eqref{neq:vdbound} we conclude that 
    $\frac{v}{d} \chi_{y(\Omega)}$
    is the weak limit of 
     $\frac{v^n}{d^n} \chi_{y^n(\Omega)}$ in $L^2(\R^d)$
     (see, for example, \cite[Theorem 13.44]{HewStr1975}).
     Thus,
     \begin{equation*}
        \int\limits_{y(\Omega)}\left|\frac{v(\xi)}{d(\xi)}\right|^2 \, \d \xi 
        \leq \liminf_{n\to \infty}
        \int\limits_{y^n(\Omega)}\left|\frac{v^n(\xi)}{d^n(\xi)}\right|^2 \, \d \xi  
        \leq C,
    \end{equation*}
    which in turn yields the thesis.
\end{proof}







\section*{Acknowledgements}
E.~D. acknowledges support from the Austrian Science Fund (FWF) through
projects F\,65,  I\,4052, V\,662, and Y1292, as well as from BMBWF through
the OeAD-WTZ project CZ04/2019. The research activity of \OOO A.~M. \EEE has been
supported by the Austrian Science Fund (FWF) \OOO projects M\,2670, I 5149, and by the OeAD-WTZ project CZ 01/2021. \UUU U.~S. is supported by the Austrian Science Fund (FWF) projects
F\,65, W\,1245,  I\,4354, I\,5149, and P\,32788, and by the OeAD-WTZ project CZ
01/2021. \EEE


\begin{thebibliography}{50}

\bibitem{Ahl1963}
L.~V.~Ahlfors.
\newblock Quasiconformal reflections.
\newblock {\em Acta Math.} \textbf{109} (1963), 291--301.

\bibitem{BalZar2017}
J.~M.~Ball, A.~Zarnescu.
\newblock Partial regularity and smooth topology-preserving approximations of
  rough domains.
\newblock {\em Calc. Var. Partial Differential Equations}, \textbf{56}(1) (2017), Paper No. 13.

\bibitem{barchiesi.desimone}
M.~Barchiesi, A.~DeSimone. Frank energy for nematic elastomers: a nonlinear model. {\em ESAIM Control Optim. Calc. Var.} \textbf{21} (2015), 372--377.

\bibitem{barchiesi.henao.moracorral}
M.~Barchiesi, D.~Henao, C.~Mora-Corral. Local invertibility in Sobolev spaces with applications to nematic elastomers and magnetoelasticity. {\em Arch. Ration. Mech. Anal.} \textbf{224} (2017), 743--816.

\bibitem{BouHenMol2019}
O.~Bouchala, S.~Hencl, A.~Molchanova.
\newblock Injectivity almost everywhere for weak limits of {S}obolev
  homeomorphisms.
\newblock {\em J. Funct. Anal.} (2020), to appear.

\bibitem{bresciani}
M.~Bresciani. Linearized Von K\'arm\'an theories for incompressible magnetoelastic plates. {\em Preprint ArXiv 2007.14122}.

\bibitem{ChaWilHewMoi2017}
S.~N. Chandler-Wilde, D.~P. Hewett, A.~Moiola.
\newblock Sobolev spaces on non-{L}ipschitz subsets of {$\Bbb{R}^n$} with
  application to boundary integral equations on fractal screens.
\newblock {\em Integral Equations Operator Theory}, \textbf{87}(2) (2017), 179--224.

\bibitem{dacorogna.fonseca}
B.~Dacorogna, I.~Fonseca. A minimization problem involving variation of the domain. {\em Comm. Pure Appl. Math.} \textbf{45} (1992), 871--897.

\bibitem{DalMaso}
G.~Dal~Maso.
\newblock \UUU Lecture notes on capacity. Unpublished.  \EEE

\bibitem{davoli.kruzik.pelech}
E.~Davoli, M.~Kru\v{z}\'ik, P.~Pelech.
Separately Global Solutions to Rate-Independent Processes in Large-Strain Inelasticity. {\em Preprint ArXiv 2008.02244}.


\bibitem{davoli.kruzik.piovano.stefanelli}
E.~Davoli, M.~Kru\v{z}\'ik, P.~Piovano, U.~Stefanelli. Magnetoelastic thin films at large strains.
{\em Continuum Mechanics and Thermodynamics} \MMM \textbf{33}, (2021) 327--341. \EEE

\bibitem{EdmEva1987}
D. E. Edmunds, W. D. Evans.
{\it Spectral Theory and Differential Operators}. 
Oxford University Press, Oxford, 1987.

\bibitem{EdmNek2017}
D. E. Edmunds,  A. Nekvinda.
Characterisation of zero trace functions in variable exponent {S}obolev spaces.
\textit{Math. Nachr.}
\textbf{290}(14-15) (2017), 2247--2258.

\bibitem{EkeTem1976}
I.~Ekeland and R.~Temam.
\newblock {\em Convex analysis and variational problems}.
\newblock North-Holland, Amsterdam, 1976.

\bibitem{fonseca.parry}
I.~Fonseca, G.~Parry. Equilibrium  configurations of defective crystals. {\em Arch. Ration. Mech. Anal.} \textbf{120} (1992), 245--283.

\bibitem{GehIwa1999}
F.~W. Gehring, T.~Iwaniec.
\newblock The limit of mappings with finite distortion.
\newblock {\em Ann. Acad. Sci. Fenn. Math.}, \textbf{24} (1999), 253--264.


\UUU
\bibitem{Novaga0}
M.~Goldman, M.~Novaga, \OOO B.~Ruffini. \EEE Existence and stability for a
non-local isoperimetric model of charged liquid drops. {\it
  Arch. Ration. Mech. Anal.} \textbf{217}(1) (2015),   1--36.

\bibitem{Ortigosa0}
A.~J.~Gil, R.~Ortigosa. A new framework
for large strain electromechanics based on convex multi-variable
strain energies: variational formulation and material
characterisation. {\it Comput. Methods Appl. Mech. Engrg.} \textbf{302} (2016),
293--328.
\EEE

\bibitem{GraKruMaiSte2019}
D.~Grandi, M.~Kru\v{z}\'{\i}k, E.~Mainini, U.~Stefanelli.
\newblock A phase-field approach to {E}ulerian interfacial energies.
\newblock {\em Arch. Ration. Mech. Anal.} \textbf{234}(1) (2019), 351--373.

\bibitem{GraKruMaiSte2019-2}
D.~Grandi, M.~Kru\v{z}\'{\i}k, E.~Mainini, U.~Stefanelli. Equilibrium for Multiphase Solids with Eulerian Interfaces. {\em J. Elast.} \textbf{142} (2020), 409--431. 

\bibitem{Haj1999}
P.~Haj{\l}asz.
\newblock Pointwise Hardy inequalities.
\newblock {\em Proc. Amer. Math. Soc.}
\textbf{127}(2) (1999), 417--423.

\bibitem{HarHasKos2005}
P.~Harjulehto, P. H\"{a}st\"{o}, M. Koskenoja. \newblock Hardy's inequality in a variable exponent Sobolev space.
\newblock {\em Georgian Math. J.}
\textbf{12}(3) (2005), 431--442.

\bibitem{HeiKilMar2006}
J.~Heinonen, T.~Kilpel\"{a}inen, O.~Martio.
\newblock {\em Nonlinear potential theory of degenerate elliptic equations}.
\newblock Dover Publications, Inc., Mineola, NY, 2006.






\bibitem{HenKos2014}
S.~Hencl, P.~Koskela.
\newblock {\em Lectures on mappings of finite distortion}.
\newblock {\em Adv. Calc. Var.} \textbf{11}(1) (2018), 65--73.
  Publishing, 2014.



\bibitem{HewStr1975}
E.~Hewitt, K.~Stromberg. 
\newblock {\em Real and Abstract Analysis}. \newblock Springer Verlag, 1975.

\bibitem{IwaMar2001}
T.~Iwaniec, G.~Martin.
\newblock {\em Geometric function theory and non-linear analysis}.
\newblock Oxford Mathematical Monographs, Clarendon Press, Oxford, 2001.

\bibitem{IwaSve1993}
T.~Iwaniec, V.~{\v S}ver{\' a}k.
\newblock On mappings with integrable dilatation.
\newblock {\em Proc. Amer. Math. Soc.} \textbf{118} (1993), 185--188.

\bibitem{javili.mcbride.steinmann}
A.~Javili, A.~McBride, P.~Steinmann. Thermomechanics of solids with
lower-dimensional energetics: on the importance of surface, interface,
and curve structures at the nanoscale. A unifying review. {\em Appl. Mech. Rev.} \textbf{65} (2013), 010802.

\bibitem{JerKen1982}
D.~S.~Jerison, C.~E.~Kenig.
\newblock Boundary behavior of harmonic functions in nontangentially accessible domains.
\newblock {\em Adv. in Math.} \textbf{46}(1) (1982), 80--147.

\bibitem{KadKuf1999}
J.~Kadlec, A.~Kufner. 
\newblock Characterisation of functions with zero traces by integrals with weight functions. 
\newblock {\em Proc. Amer. Math. Soc.} \textbf{127} (1999), 417-–423.


\bibitem{KinMar1997}
J.~Kinnunen, O.~Martio. 
\newblock Hardy’s inequalities for Sobolev functions.
\newblock {\em Math. Res. Lett.} \textbf{4} (1997), 489-–500.

\bibitem{KosMal2003}
P.~Koskela, J.~Mal{\' y}.
\newblock Mappings of finite distortion: the zero set of the {J}acobian.
\newblock {\em J. Eur. Math. Soc.} \textbf{5} (2003), 95--105.

\bibitem{kruzik.melching.stefanelli}
M.~Kru\v{z}\'ik, D.~Melching, U.~Stefanelli. Quasistatic evolution for dislocation-free finite plasticity. {\em ESAIM Control Optim. Calc. Var.} \textbf{26} (2020), Paper No. 123.

\bibitem{kruzik.stefanelli.zeman}
M.~Kru\v{z}\'ik, U.~Stefanelli, J.~Zeman. Existence results for incompressible magnetoelasticity. {\em Discrete Contin. Dyn. Syst.} \textbf{35} (2015), 2615--2623.

\bibitem{levitas}
V.~I.~Levitas. Phase field approach to martensitic phase transformations with large strains and interface
stresses. {\em J. Mech. Phys. Solids}, \textbf{70} (2014), 154--189. 

\bibitem{liakhova}
J.~Liakhova. {\em A theory of magnetostrictive thin films with applications}. Ph.D. Thesis, University of Minnesota, 1999.

\bibitem{liakhova.luskin.zhang}
\UUU J.~Liakhova, \EEE M.~Luskin, T.~Zhang. Computational modeling of ferromagnetic shape memory thin films. {\em Ferroelectrics} \textbf{342} (2006), 7382.


\bibitem{LieLos2001}
E.~H.~Lieb, M.~Loss.
\newblock {\em Analysis}, volume~14 of {\em Graduate Studies in Mathematics}.
\newblock American Mathematical Society, Providence, RI, second edition, 2001.

\bibitem{luskin.zhang}
M.~Luskin, T.~Zhang. Numerical analysis of a model for ferromagnetic shape memory thin films. {\em Comput. Methods Appl. Mech. Engrg.} \textbf{196} (2007), 37--40.

\bibitem{ManVill1998}
J.~Manfredi, E.~Villamor.
\newblock An extension of {R}eshetnyak's theorem.
\newblock {\em Indiana Univ. Math. J.} \textbf{47}(3) (1998), 1131--1145.

\bibitem{MarMiz1973}
M.~Marcus, V.~J.~Mizel.
\newblock Transformations by functions in {S}obolev spaces and lower
  semicontinuity for parametric variational problems.
\newblock {\em Bull. Amer. Math. Soc.} \textbf{79} (1973), 790--795.

\bibitem{May2010}
D.~Meyer.
\newblock Snowballs are quasiballs.
\newblock {\em Trans. Amer. Math. Soc.} \textbf{362}(3) (2010), 1247--1300.

\bibitem{MolVod2020}
A.~Molchanova, S.~Vodopyanov.
\newblock Injectivity almost everywhere and mappings with finite distortion in
  nonlinear elasticity.
\newblock {\em Calc. Var. Partial Differential Equations},
\textbf{59}(1) (2020): Paper No. 17.

\bibitem{MulSpe1995}
S.~M\"uller, S.~Spector.
\newblock An existence theory for nonlinear elasticity that allows for
  cavitation.
\newblock {\em Arch. Ration. Mech. Anal.} \textbf{131}(1) (1995),
1--66.

\UUU


\bibitem{Novaga00}
 C.~B.~Muratov, M.~Novaga. On well-posedness of variational models
  of charged drops. {\it Proc. A}, \textbf{472}(2187) (2016), 20150808, 12 pp.
  
\bibitem{Novaga1}
C.~B.~Muratov, M.~Novaga, B.~Ruffini. On equilibrium
  shape of charged flat drops. {\it Comm. Pure Appl. Math.}
  \textbf{71}(6) (2018),  1049--1073.
  
\bibitem{Novaga2}
C.~B.~Muratov, M.~Novaga, B.~Ruffini.
  Conducting flat drops in a confining potential. \MMM Preprint ArXiv:2006.02839. \UUU





\bibitem{Ortigosa1}
R.~Ortigosa, A.~J.~Gil. 
A new framework for large strain electromechanics based on convex
multi-variable strain energies: finite element discretisation and
computational implementation. {\it Comput. Methods Appl. Mech. Engrg.} \textbf{302}
(2016), 329--360.

\bibitem{Ortigosa2}
R.~Ortigosa, A.~J.~Gil. A new framework for large strain
electromechanics based on convex multi-variable strain energies:
conservation laws, hyperbolicity and extension to
electro-magneto-mechanics. {\it Comput. Methods Appl. Mech. Engrg.}
\textbf{309} (2016), 202--242.



  \EEE

\bibitem{Resh1967-2}
Y.~G.~Reshetnyak.
\newblock Space mappings with bounded distortion.
\newblock {\em Sib. Math. J.} \textbf{8}(3) (1967), 466--487.

\bibitem{Resh1982}
Y.~G.~Reshetnyak.
\newblock {\em Space mappings with bounded distortion}.
\newblock Transl. Math. Monographs 73, AMS, New York, 1989.

\bibitem{Rick1993}
S.~Rickman.
\newblock {\em Quasiregular mappings}.
\newblock Springer-Verlag, Berlin, 1993.

\bibitem{rybka.luskin}
P.~Rybka, M.~Luskin. Existence of energy minimizers for magnetostrictive materials. {\em SIAM J. Math.
  Anal.} \textbf{36} (2005), 2004--2019.

\UUU

\bibitem{Silhavy1}
 M.~\v Silhav\' y. A variational approach to
 electro-magneto-elasticity: Convexity conditions and existence
 theorems. {\it Math. Mech.
   Solids}, \textbf{23}(6) (2018), 907--928.

\bibitem{Silhavy2}
M.~\v Silhav\' y.  Isotropic polyconvex electromagnetoelastic
bodies. {\it Math. Mech. Solids}, \textbf{24}(3), (2019),  738--747.

 \EEE


\bibitem{stefanelli}
U.~Stefanelli. Existence for dislocation-free finite plasticity. {\em ESAIM Control Optim. Calc. Var.} \textbf{25} (2019), Paper No. 21.

\bibitem{Vod2012}
S.~K.~Vodop$'$yanov.
\newblock Regularity of mappings inverse to {S}obolev mappings.
\newblock {\em Mat. Sb.} \textbf{203}(10) (2012), 1383--1410.

\bibitem{VodGold1976}
S.~K.~Vodop$'$yanov, V.~M.~Gol$'$dshtein.
\newblock Quasiconformal mappings and spaces of functions with generalized
  first derivatives.
  \newblock {\em Sib. Math. J.} \textbf{17}(3) (1976), 399--411.

\OOO
\bibitem{VodMol2016}
S.~K.~Vodop$'$yanov, A.~O.~Molchanova.
\newblock Lower semicontinuity of mappings with bounded $(\theta,1)$-weighted $(p,q)$-distortion.
\newblock {\em Sib. Math. J.} \textbf{57}(5) (2016), 778--787.

\UUU
\bibitem{Wang}
  J.~Wang, M.-F. Lin, S. Park, P.  See Lee,
Deformable conductors for human-machine interface. {\it Materials
  Today}, \textbf{21}(5) (2018), 508--526.

\end{thebibliography}

\end{document}